\documentclass[11pt, a4paper]{amsart}
\usepackage{graphicx}
\usepackage{tikz}
\usepackage{amssymb,amsmath,amsthm}
\usepackage{verbatim}
\usepackage{color}
\usepackage[shortlabels]{enumitem}
\usepackage{dsfont}
\usepackage[margin=2.75cm]{geometry}
\usepackage{thmtools, thm-restate}
\usepackage{hyperref}
\usepackage[capitalise]{cleveref}

\declaretheorem{theorem}
\declaretheorem[sibling=theorem]{corollary, lemma, proposition, question, definition, conjecture,remark}

\begin{document}

\title{Metric graphs of negative type}
\author{Rutger Campbell}
\address{Discrete Mathematics Group, Institute for Basic Science (IBS), Daejeon, Republic of Korea}
\email{rutger@ibs.re.kr}
\author{Kevin Hendrey}
\address{School of Mathematics, Monash University, Melbourne, Australia}
\email{kevin.hendrey1@monash.edu}
\author{Ben Lund}
\address{Discrete Mathematics Group, Institute for Basic Science (IBS), Daejeon, Republic of Korea}
\email{benlund@ibs.re.kr}
\author{Casey Tompkins}
\address{HUN-REN Alfr\'ed R\'enyi Institute of mathematics}
\email{casey.tompkins@renyi.hu}

\thanks{All authors were supported by the Institute for Basic Science (IBS-R029-C1). The second author was also supported by the Australian Research Council. The fourth author was also supported by National Research, Development and Innovation Office, NKFIH, grants K135800 and K132696.}

\begin{abstract}
    The negative type inequalities of a metric space are closely tied to embeddability.
    A result by Gupta, Newman, and Rabinovich implies that if a metric graph $G$ does not contain a theta submetric as an embedding, then $G$ has negative type.
    We show the converse: if a metric graph $G$ contains a theta, then it does not have negative type.
\end{abstract}

\maketitle

\section{Introduction}

Let $(M,d)$ be a metric space.
For a function $\omega:M \rightarrow \mathbb{R}$ with finite support $X:=\mathop{supp}(\omega)$, denote
\[\gamma_d(\omega) = \sum_{x,y \in X} \omega(x)\omega(y)d(x,y).\]
A metric space has {\em negative type} if $\gamma_d(\omega)$ is negative for all such $\omega$ with $\sum_{x \in \mathop{supp}(\omega)} \omega(x) = 0$.
Negative type is a classic concept used for embeddability; recall an embedding is a map $f$ from $(M,d)$ to $(M',d')$ where $d(x,y)=d'(f(x),f(y))$ for any $x,y\in M$.
Notably, Schoenberg~\cite{schoenberg1935} showed that $(M,d)$ has negative type if and only if $(M,\sqrt{d})$ is embeddable in Euclidean space.
The following result of Deza, Laurent and Weismantel~\cite{deza1997geometry} summarizes the relationship between negative type and some closely related properties.
\begin{theorem}[{\cite[Theorem~6.3.1]{deza1997geometry}}]
Let $(X,d)$ be a finite metric space.
We label the following conditions:
\begin{enumerate}[(i)]
    \item $(X,d)$ is embeddable in some $(\mathbb{R}^n,d_2)$ where $d_2(x,y)=\left(\sum_{i=1}^n |x_i-y_i|^2\right)^{\frac{1}{2}}$.
    \item $(X,d)$ is embeddable in some $(\mathbb{R}^n,d_1)$ where $d_1(x,y)=\sum_{i=1}^n |x_i-y_i|$.
    \item $(X,d)$ has negative type.
    \item $(X,\sqrt{d})$ is embeddable in some $(\mathbb{R}^n,d_2)$ where $d_2(x,y)=\left(\sum_{i=1}^n |x_i-y_i|^2\right)^{\frac{1}{2}}$.
    \item The distance matrix of $(X,d)$ has exactly one positive eigenvalue.
\end{enumerate}
We have the following chain of implications:
$(i)\Rightarrow(ii)\Rightarrow(iii)\Leftrightarrow(iv)\Rightarrow(v)$.
\end{theorem}
Property $(ii)$ is called $\ell_1$-embeddable.

We wish to consider metric spaces that resemble a network (of roads, for instance).
A {\em metric segment} of length $\ell$ is a metric space isometric to the real interval $[0,\ell]$.
A {\em metric graph} $G=(M,d)$ is the result of gluing a disjoint union of finitely many metric segments $E$ and points $V$ using an equivalence relation $R$ on the endpoints of elements of $E$ and elements $V$, where there is exactly one element of $V$ in each equivalence class of $R$.
We call $E$ the set of {\em edges} and $V$ the set of {\em vertices} of the metric graph and denote these $E(G)$ and $V(G)$ respectively.
The metric on $G$ is the maximal metric bounded above by the metrics on each element of $E$ and by the metric $d_R(a,b)$ that is $0$ when $aRb$ and infinity otherwise.

A \emph{theta} is any subset of points that is isometric to the metric graph consisting of two vertices and three metric segments $E_1,E_2,E_3$ that all have an endpoint in both $R$-equivalence classes.
A metric graph is \emph{theta-containing} if there is an embedding from a theta, and otherwise is \emph{theta-free}.

\newpage
\begin{theorem}[\cite{gupta2004cuts}]\label{th:weightedGraphs}
    For any graph $G$, the following are equivalent:
    \begin{enumerate}
        \item for every weighting $d$ of the edges of $G$, the shortest path metric on the vertices of $G$ induced by $d$ is $\ell_1$-embeddable,
        \item $G$ does not contain $K_{2,3}$ as a minor.
    \end{enumerate}
\end{theorem}

This theorem implies that all theta-free metric graphs have negative type.
Indeed, suppose that $G=(M,d)$ is a theta-free metric graph, and let $X$ be a finite subset of $M$. Then the graph with vertex set $X$ and edge lengths determined by $d$ does not contain a $K_{2,3}$-minor, and hence $(X,d)$ is $\ell_1$-embeddable, and hence has negative type.
Since this is true for every finite submetric of $G$, it follows that $G$ has negative type.

On the other hand, \cref{th:weightedGraphs} does not directly say anything 
about theta-containing metric graphs, as it does not say anything about graphs with predetermined edge weights.%
Note, for example, that complete graphs with only weight $1$ edges are $\ell_1$-embeddable: just map each vertex $x$ to its own basis vector $e_x$.
Our main result implies that theta-containing metric graphs do not have negative type.

\begin{theorem}\label{th:uniformNegTypeBound}
    If $G$ is a metric graph that contains a theta, and each segment of $E(G)$ has length at least $1$, then there is a set $B$ of three points and a set $R$ of three points such that
    \[ \sum_{x,y \in R} d(x,y) + \sum_{x,y \in B} d(x,y) - \sum_{x \in R, y \in B} d(x,y) \geq 1/12.\]
\end{theorem}

With \cref{th:weightedGraphs}, this gives a complete classification of $\ell_1$-embeddable metric graphs.
\begin{corollary}
    A metric graph is $\ell_1$-embeddable if and only if it is theta-free.
\end{corollary}

\subsection{Negative type gap for theta-containing metric graphs}
Let $(M,d)$ be a metric space, and denote
\[\Gamma(M,d) = \sup_{\omega} \gamma_d(\omega),\]
where $\omega$ ranges over functions with finite support with $\sum_{x \in M} \omega(x) = 0$ and $\sum_{x \in M} |\omega(x)| = 1$.
$\Gamma(M,d)$ is the {\em negative type gap} of $(M,d)$.
The notion of negative type gap was introduced in~\cite{doust2008enhanced}.

For a family $\mathcal{M}$ of metric spaces, denote
\[
\Gamma_{\inf} (\mathcal{M}) = \inf_{(M,d) \in \mathcal{M}} \Gamma(M,d).\]
Note that, for any metric space $(X,d)$ and $t>0$, we have $\Gamma((X,d)) = t^{-1}\Gamma((X,td))$.
Consequently, it is only interesting to consider $\Gamma_{\inf}(\mathcal{M})$ if we fix a scale in the definition of $\mathcal{M}$.

\cref{th:uniformNegTypeBound} immediately implies the following lower bound on $\Gamma_{\inf}$ for an interesting family of metric graphs.

\begin{corollary}
    If $\mathcal{G}$ is the family of theta-containing metric graphs having all edge lengths at least $1$,
    then
    \[\Gamma_{\inf}(\mathcal{G}) \geq 1/432. \]
\end{corollary}

\begin{proof}
    This is witnessed by the six points given by \cref{th:uniformNegTypeBound}, with $\omega(x) = -(1/6)$ for $x \in B$ and $\omega(y) = 1/6$ for $y \in R$.
\end{proof}

Note that in a theta $(M,d)$ consisting of three metric segments of length 1, the distance between any two points is at most 1.
Thus for $\omega$ with $\sum_{x\in M} |\omega(x)|=1$, we have $\gamma_d(\omega)=\sum_{x\in M} \omega(x)\omega(y)d(x,y)\leq 1$.
This gives us that $\Gamma_{\inf}(\mathcal{G})\leq 1$ for any family containing such a theta.

We leave the following question open.
\begin{question}
    What is $\Gamma_{\inf}(\mathcal{G})$ for the family of theta-containing metric graphs having all edge lengths at least $1$.
\end{question}

\subsection{Subdivisions of graphs}

\Cref{th:uniformNegTypeBound} also has the following consequence for the shortest path metric on subdivisions of graphs that contain $K_{2,3}$ as a minor.

\begin{corollary}
    Let $G$ be a graph that contains $K_{2,3}$ as a minor.
    Then, the shortest path metric of the $180$-subdivision of $G$ does not have negative type.
\end{corollary}
\begin{proof}
    Let $G'$ be the $180$-subdivision of $G$.
    Let $M$ be the metric graph obtained by taking $V$ to be the vertices of degree at least $3$ in $G'$, and segments $E$ for induced paths in $G'$, with length equal to the corresponding path length in $G'$.
    Since $G$ has a $K_{2,3}$-minor, $M$ contains a theta.
    Since each segment of $M$ has length at least $181$, \cref{th:uniformNegTypeBound} implies that there are sets $R$ and $B$ of three points each such that
    \[\sum_{x,y \in R} d(x,y) + \sum_{x,y \in B} d(x,y) - \sum_{x \in R, y \in B} d(x,y) \geq \frac{181}{12}  > 15. \]
    Each of these points must be at distance at most $1/2$ from the location of an actual vertex of $G'$.
    Hence, we have sets $R',B' \subset V(G')$ such that $d(x',y') + 1 \geq d(x,y) \geq d(x',y') - 1$ for any  $x,y \in R \cup B$ and $x',y' \in R' \cup B'$.
    Hence,
    \[\sum_{x,y \in R'} d(x,y) + \sum_{x,y \in B'} d(x,y) - \sum_{x \in R', y \in B'} d(x,y) > 0,
    \]
    as claimed.
\end{proof}
A better bound than $180$ can be recovered directly from the proof of \cref{th:uniformNegTypeBound}, but such a result would most likely be far from best possible.

\begin{proposition}
    The shortest path metric of the $2$-subdivision of $K_4$ is $\ell_1$-embeddable.
\end{proposition}
\begin{proof}
For a vertex set $S$ in a graph $G$, the cut metric $d_S$ is defined by $d_S(x,y)=1$ if exactly one of $x$ and $y$ is in $S$ and $d_S(x,y)=0$ otherwise.  
It is known~\cite{deza1997geometry} that a graph $G$ is $\ell_1$-embeddable if and only if the shortest path metric $d$ is equal to a linear combination of cut metrics on $G$ with positive coefficients. 
Let $x_1, x_2,x_3$ and $x_4$ be the degree-$3$ vertices of the $2$-subdivision of $K_4$. 
For $i,j\in [4]$, let $S_{i,j}$ be the set containing $x_i$ and all vertices within distance $2$ of $x_i$ according to $d$ which are not adjacent to $x_j$.
It is simple to verify that $2d$ is equal to the sum of all $d_{S_{i,j}}$.
\end{proof}

In light of the above theorems, we propose the following conjecture.
\begin{conjecture}
    There is no graph $G$ with a $K_{2,3}$ minor where the shortest distance metric of the $3$-subdivision of $G$ has negative type.
\end{conjecture}

\section{Proof of main theorem}

Let $T$ be a theta with minimum total length in $G$, let $u,v$ be the two vertices of degree three in $T$, and let $P_1,P_2,P_3$ be the three internally disjoint $uv$ paths of $T$, with $|P_1| \leq |P_2| \leq |P_3|$.
For $x,y \in T$, let $d_T(x,y)$ be the distance between $x$ and $y$ in $T$.

\begin{lemma}\label{th:shortcuts}
    Let $S$ be an $xy$-path in $G$, internally disjoint from $T$, where $x \in P_i$ and $y \in P_j$ with $i<j$.
    Then, 
    \[|S| \geq \max(d_T(u,x),d_T(u,y),d_T(v,x),d_T(v,y)).\]
\end{lemma}
\begin{proof}
    Let $P \subset P_i$ be a $xu$-path.
    Then $T \setminus P \cup S$ is a theta - see \cref{fig:thetaShortcut}.
    Since $T$ has minimum total length, we have
    \[|S| \geq |P| \geq d_T(x,u). \]
    The same reasoning applies to the other named distances.
\end{proof}

\begin{figure}[h!]
\begin{center}
\begin{tikzpicture}[scale=0.6]
    \draw[black, thick] (0,0)--(4,0);
    \draw[black, thick] (0,0) arc [x radius = 2, y radius =3, start angle = 180, end angle =0];
    \draw[black, thick] (0,0) arc [x radius = 2, y radius =1.5, start angle = -180, end angle =0];
    \draw[blue!50, thick] (1,0)--(2,3);

    \draw[black, thick] (5,0)--(9,0);
    \draw[black, thick] (5,0) arc [x radius = 2, y radius =1.5, start angle = -180, end angle =0];
    \draw[black, thick] (9,0) arc [x radius = 2, y radius =3, start angle = 0, end angle =90];
    \draw[blue!50, thick] (6,0)--(7,3);

    \draw[black, thick] (11,0)--(14,0);
    \draw[black, thick] (10,0) arc [x radius = 2, y radius =3, start angle = 180, end angle =0];
    \draw[black, thick] (10,0) arc [x radius = 2, y radius =1.5, start angle = -180, end angle =0];
    \draw[blue!50, thick] (11,0)--(12,3);

    \draw[black, thick] (15,0)--(16,0);
    \draw[black, thick] (15,0) arc [x radius = 2, y radius =3, start angle = 180, end angle =0];
    \draw[black, thick] (15,0) arc [x radius = 2, y radius =1.5, start angle = -180, end angle =0];
    \draw[blue!50, thick] (16,0)--(17,3);

    \draw[black, thick] (20,0)--(24,0);
    \draw[black, thick] (20,0) arc [x radius = 2, y radius =3, start angle = 180, end angle =90];
    \draw[black, thick] (20,0) arc [x radius = 2, y radius =1.5, start angle = -180, end angle =0];
    \draw[blue!50, thick] (21,0)--(22,3);
\end{tikzpicture}
\end{center}\caption{A black theta with a blue shortcut, followed by four alternate thetas formed by replacing a path in the theta by the shortcut. If the black theta is minimal, then the blue shortcut is at least as long as each path removed to obtain the alternate thetas.}
\label{fig:thetaShortcut}
\end{figure}
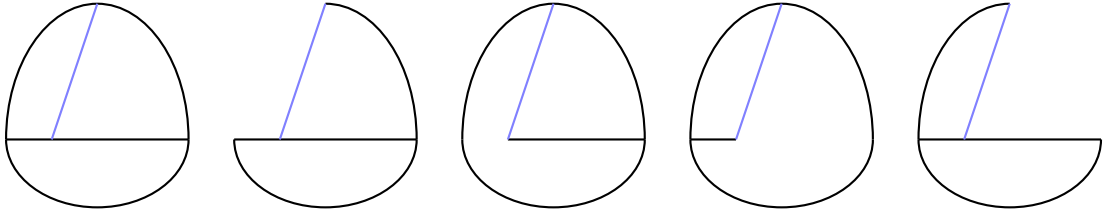

\begin{lemma}\label{th:distancePreservedFromBranchPoints}
    If $x,y \in T$ and $d(x,u) \leq 1/2$, then $d(x,y) = d_T(x,y)$.
\end{lemma}
\begin{proof}
    Suppose that $d(x,y) < d_T(x,y)$.
    Let $Q$ be an $xy$-path with $|Q| = d(x,y)$.
    Suppose $y$ is the only point of $Q \cap T$ such that $d(x,y) < d_T(x,y)$; if not, rename $y$ and $Q$ so that this holds.
    Then, there is a point $z \in T$ such that $Q$ consists of an $xz$-path of length $d_T(x,z)$, followed by a $zy$-path $S$ internally disjoint from $T$ such that $|S| < d_T(z,y)$.
    
    If $z,y \in P_i$ for $i \in [3]$, then the minimality of $T$ implies that $|S| \geq d_T(z,y)$, a contradiction.
    Otherwise, using \cref{th:shortcuts} and the assumption that $d(u,x) \leq 1/2 \leq d(x,z)$, we have
    \[d_T(x,y) \leq d_T(x,u) + d_T(u,y) \leq d_T(x,u) + |S| \leq d(x,z) + |S| = d(x,y), \]
    a contradiction.
\end{proof}

\begin{proof}[Proof of \cref{th:uniformNegTypeBound}]
Let $I_1$ be the closed interval of length $1/2$ such that $u \in I_1 \subset P_1$.
For $x \in I_1$, define $\phi_2(x)$ to be the point opposite to $x$ in the cycle $P_1 \cup P_2$, and similarly define $\phi_3(x)$ to be the point opposite to $x$ in the cycle $P_1 \cup P_3$.
Since $|P_1| \leq |P_2|$, the image $I_2$ of $\phi_2$ is contained in $P_2$, and likewise the image $I_3$ of $\phi_3$ is contained in $P_3$.

Since $|I_2| = |I_3| = 1/2$, there is at most one vertex in $I_2$ and at most one vertex in $I_3$.
Hence, by an averaging argument, there are closed intervals $J_1 \subset I_1, J_2 \subset I_2$, and $J_3 \subset I_3$ such that $|J_1|=1/6$, the image of $J_1$ under $\phi_j$ is $J_j$ for $j \in \{2,3\}$, and neither $J_2$ nor $J_3$ has any interior vertices.
Note that $I_2 \cap I_3 = \emptyset$ unless $|P_1| = |P_2| = |P_3|$, and in this case the only vertices in $I_2$ and $I_3$ occur at their shared endpoint. Hence, we may assume that $J_2 \cap J_3 = \emptyset$.

Let $x_1,x_2,x_3$ be evenly spaced points in $J_1$, with $x_1$ and $x_3$ the endpoints.
For $i \in [3]$, let $y_i = \phi_2(x_i)$ and $z_i = \phi_3(x_i)$.

Let $i \in [2]$.
Using the fact, by \cref{th:distancePreservedFromBranchPoints}, that $d(x,y)=d_T(x,y)$ for $x \in I_1$ and $y \in T$, and the definition of $\phi_2$, we have
\[d(x_i,y_i) = d(x_i,y_{i+1}) + d(y_i,y_{i+1}) = d(x_i,y_{i+1}) + 1/12.\]
The same reasoning applies to $d(x_{i+1},y_{i+1}), d(x_i,z_i)$, and $d(x_{i+1},z_{i+1})$, and hence
\begin{multline}\label{eq:xyxz}
d(x_i,y_i) + d(x_{i+1},y_{i+1}) + d(x_i,z_i) + d(x_{i+1},z_{i+1}) = \\ d(x_i,y_{i+1}) + d(x_{i+1},y_i) + d(x_i, z_{i+1}) + d(x_{i+1},z_i) + 4/12. \end{multline}

It remains to consider distances between points in $J_2$ and $J_3$.
In particular, we will find an index $i \in [2]$ such that 
\begin{equation}\label{eq:yz} d(y_i,z_i) + d(y_{i+1},z_{i+1}) \geq d(y_i,z_{i+1}) + d(y_{i+1},z_i). \end{equation}

Since there are no vertices in the interior of $J_2$ or $J_3$, we have
\[d(y_2,z_2) = 1/6 + \min(d(y_1,z_1),d(y_1,z_3),d(y_3,z_1),d(y_3,z_3)); \]
see \cref{fig:y2z2paths}.

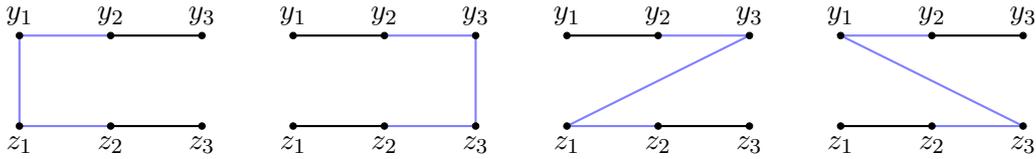
\begin{figure}[h!]
\begin{center}
\begin{tikzpicture}[scale=0.6]

\draw[black, thick] (2,0)--(4,0);
\draw[black, thick] (2,2)--(4,2);
\draw[blue!50, thick] (2,2)--(0,2)--(0,0)--(2,0);
\filldraw[black] (2,0) circle (2pt) node[anchor=north]{$z_2$};
\filldraw[black] (0,0) circle (2pt) node[anchor=north]{$z_1$};
\filldraw[black] (4,0) circle (2pt) node[anchor=north]{$z_3$};
\filldraw[black] (2,2) circle (2pt) node[anchor=south]{$y_2$};
\filldraw[black] (0,2) circle (2pt) node[anchor=south]{$y_1$};
\filldraw[black] (4,2) circle (2pt) node[anchor=south]{$y_3$};

\draw[black, thick] (6,0)--(8,0);
\draw[black, thick] (6,2)--(8,2);
\draw[blue!50, thick] (8,2)--(10,2)--(10,0)--(8,0);
\filldraw[black] (8,0) circle (2pt) node[anchor=north]{$z_2$};
\filldraw[black] (6,0) circle (2pt) node[anchor=north]{$z_1$};
\filldraw[black] (10,0) circle (2pt) node[anchor=north]{$z_3$};
\filldraw[black] (8,2) circle (2pt) node[anchor=south]{$y_2$};
\filldraw[black] (6,2) circle (2pt) node[anchor=south]{$y_1$};
\filldraw[black] (10,2) circle (2pt) node[anchor=south]{$y_3$};

\draw[black, thick] (14,0)--(16,0);
\draw[black, thick] (12,2)--(14,2);
\draw[blue!50, thick] (14,2)--(16,2)--(12,0)--(14,0);
\filldraw[black] (14,0) circle (2pt) node[anchor=north]{$z_2$};
\filldraw[black] (12,0) circle (2pt) node[anchor=north]{$z_1$};
\filldraw[black] (16,0) circle (2pt) node[anchor=north]{$z_3$};
\filldraw[black] (14,2) circle (2pt) node[anchor=south]{$y_2$};
\filldraw[black] (12,2) circle (2pt) node[anchor=south]{$y_1$};
\filldraw[black] (16,2) circle (2pt) node[anchor=south]{$y_3$};

\draw[black, thick] (20,2)--(22,2);
\draw[black, thick] (18,0)--(20,0);
\draw[blue!50, thick] (20,0)--(22,0)--(18,2)--(20,2);
\filldraw[black] (20,0) circle (2pt) node[anchor=north]{$z_2$};
\filldraw[black] (18,0) circle (2pt) node[anchor=north]{$z_1$};
\filldraw[black] (22,0) circle (2pt) node[anchor=north]{$z_3$};
\filldraw[black] (20,2) circle (2pt) node[anchor=south]{$y_2$};
\filldraw[black] (18,2) circle (2pt) node[anchor=south]{$y_1$};
\filldraw[black] (22,2) circle (2pt) node[anchor=south]{$y_3$};

\end{tikzpicture}
\end{center}\caption{The blue paths show the possible shortest paths between $y_2$ and $z_2$.}
\label{fig:y2z2paths}
\end{figure}

If $d(y_2,z_2) = 1/6+d(y_1,z_1)$, then we have
\[d(y_1,z_1)+d(y_2,z_2) = 2d(y_1,z_1) + d(y_2,y_1)+d(z_1,z_2) \geq d(y_1,z_2) + d(y_2,z_1), \]
which establishes \cref{eq:yz} for $i=1$.
The case that $d(y_2,z_2) = 1/6+d(y_3,z_3)$ is handled similarly, and leads to the choice $i=2$.

Suppose that $d(y_2,z_2) = 1/6+d(y_1,z_3)$.
Then,
\[d(y_1,z_1) + d(y_2,z_2) = d(y_2,y_1) + d(y_1,z_1)  + d(y_1,z_3) + d(z_3,z_2) \geq d(y_2,z_1) + d(y_1,z_2),\]
which establishes \cref{eq:yz} for $i=1$.
The case $d(y_2,z_2) = 1/6+d(y_3,z_1)$ is handled similarly, and in this case either choice for $i$ works.

From \cref{eq:xyxz} and \cref{eq:yz}, we have, for the appropriate choice of $i$,
\begin{multline*}
    d(x_i,y_i) + d(x_{i+1},y_{i+1}) + d(x_i,z_i) + d(x_{i+1},z_{i+1}) + d(y_i,z_i) + d(y_{i+1},z_{i+1}) \geq \\
    d(x_i,y_{i+1}) + d(x_{i+1},y_i) + d(x_i, z_{i+1}) + d(x_{i+1},z_i) + 4/12 + d(y_i,z_{i+1}) + d(y_{i+1},z_i) = \\
    d(x_i,y_{i+1}) + d(x_{i+1},y_i) + d(x_i + z_{i+1}) + d(x_{i+1},z_i) +  d(y_i,z_{i+1}) + d(y_{i+1},z_i) + \\ d(x_i,x_{i+1})+d(y_i,y_{i+1}) + d(z_i,z_{i+1}) + 1/12,
\end{multline*}
which completes the proof of the theorem for $B=\{x_i,y_i,z_i\}$ and $R=\{x_{i+1},y_{i+1},z_{i+1}\}$.
\end{proof}

\bibliographystyle{plain}
\bibliography{negativeType}

\end{document}